\documentclass[12pt,centertags,oneside]{amsart}
\usepackage{amsmath,amssymb}
\usepackage{a4wide}
\usepackage[mathscr]{eucal}
\usepackage{mathrsfs}
\usepackage{typearea}
\usepackage{pdfsync}
\usepackage{url}

\usepackage[colorlinks=true,citecolor=blue,linkcolor=blue]{hyperref}
\usepackage{color}
\usepackage{tikz}
\usetikzlibrary{arrows,automata,backgrounds,calendar,mindmap,trees}

\usepackage[margin=1in]{geometry}

\numberwithin{equation}{section}

\newtheorem{theorem}{Theorem}[section]
\newtheorem{definition}[theorem]{Definition}
\newtheorem{proposition}[theorem]{Proposition}
\newtheorem{corollary}[theorem]{Corollary}
\newtheorem{lemma}[theorem]{Lemma}
\newtheorem{remark}[theorem]{Remark}

\title[Global existence of reaction-diffusion systems]{Close-to-equilibrium regularity for reaction-diffusion systems}

\author{Bao Quoc Tang}
\address{Institute for Mathematics and Scientific Computing,\\
University of Graz, Austria}
\email{quoc.tang@uni-graz.at}

\begin{document}
\begin{abstract} 
The close-to-equilibrium regularity of solutions to a class of reaction-diffusion systems is investigated. The considered systems typically arise from chemical reaction networks and satisfy a complex balanced condition. Under some restrictions on spatial dimensions ($d\leq 4$) and order of nonlinearities ($\mu =  1 + 4/d $), we show that if the initial data is close to a complex balanced equilibrium in $L^2$-norm, then classical solutions are shown global and converging exponentially to equilibrium in $L^{\infty}$-norm. Possible extensions to higher dimensions and order of nonlinearities are also discussed. The results of this paper improve the recent work [M.J. C\'aceres and J.A. Ca\~nizo, Nonlinear Analysis: TMA 159 (2017): 62-84].
\end{abstract}

\maketitle

\medskip

\noindent
{\bf Classification AMS 2010}: 35K57, 35B40, 35Q92, 80A30, 80A32

\medskip

\noindent
{\bf Keywords:} Reaction-diffusion systems; Close-to-equilibrium regularity; Chemical reaction networks; Complex balanced condition.   

\section{Introduction and main results}
This paper deals with global existence of classical solutions to a class of reaction-diffusion systems with initial data being close to an equilibrium. The considered systems typically arise from chemical reaction networks. More precisely, we consider $N$ chemical substances $S_1, \ldots, S_N$, which react in the $R$ reactions where the $r$-th reaction has the form
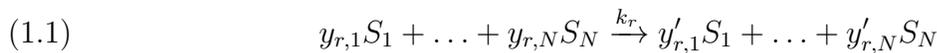
\begin{equation}\label{reactions}
	y_{r,1}S_1 + \ldots + y_{r,N}S_N \xrightarrow{k_r} y_{r,1}'S_1 + \ldots + y_{r,N}'S_N
\end{equation}
for $r=1,\ldots, R$. Here $y_r = (y_{r,1}, \ldots, y_{r,N}), y_{r}' = (y_{r,1}', \ldots, y_{r,N}')\in (\{0\}\cup [1,\infty))^N$ are stoichiometric coefficients, and $k_r>0$ is the reaction rate constant.
We will assume naturally that there exists at least one $i\in \{1, \ldots, N\}$ such that $y_{r,i} \not= y_{r,i}'$. With a slight abuse of notation we will rewrite the reaction \eqref{reactions} as
\begin{equation*}
	y_r \xrightarrow{k_r} y_r'
\end{equation*}
in which $y_r$ and $y_r'$ are called {\it complexes} with $y_r$ is the reactant and $y_r'$ is the product of the corresponding reaction. Denote by $\mathcal C = \{y_r, y_r':\, r = 1,\ldots, R\}$ the set of all complexes. Note that each complex $y\in \mathcal C$ can be a reactant, a product or both (in possibly different reactions).

To set up a reaction-diffusion system moelling \eqref{reactions} we assume that the reactions take place in a bounded vessel $\Omega \subset \mathbb R^d$ with smooth boundary $\partial\Omega$ (e.g. $C^{2+\epsilon}$ for $\epsilon>0$). For $i=1,\ldots, N$, denote by $u_i(x,t)$ the concentration $S_i$ at position $x\in\Omega$ and at time $t\ge 0$. Assume moreover that each substance $S_i$ diffuses at a positive constant rate $d_i>0$. The corresponding reaction-diffusion system for $u = (u_1, \ldots, u_N)$ reads as
\begin{equation}\label{system}
	\begin{aligned}
		\partial_t u_i - d_i\Delta u_i &= f_i(u), && x\in\Omega, \quad t>0,\\
		\nabla u_i \cdot \nu &= 0, && x\in\partial\Omega, \quad t>0,\\
		u_i(x,0) &= u_{i,0}(x), && x\in\Omega
	\end{aligned}
\end{equation}
for all $i=1,\ldots, N$, where $\nu$ is the outward normal on $\partial\Omega$ and the initial data $u_{i,0}$ are nonnegative. Here the homogeneous Neumann boundary condition means that the system is closed. The reactions $f_i(u)$ are derived from all the chemical reactions \eqref{reactions} using the {\it law of mass action}, that is for $i=1,\ldots, N$,
\begin{equation}\label{nonlinearties}
	f_i(u) = \sum_{r=1}^{R}k_r(y_{r,i}' - y_{r,i})u^{y_r} \quad \text{ with } \quad u^{y_r} = \prod_{i=1}^{N}u_i^{y_{r,i}}.
\end{equation}

Note that the nonlinearities are of polynomial type. We denote by $\mu$ the highest order of nonlinearities \eqref{nonlinearties}, i.e.
\begin{equation}\label{mu}
	\mu = \max_{y\in \mathcal C}|y| \quad \text{ with } \quad |y| = \sum_{i=1}^{N}|y_i| \text{ for any } y\in \mathbb R^N.
\end{equation}
Obviously we are interested in the case superlinear, i.e. $\mu > 1$.
It follows straightforwardly there exists $K>0$ such that
\begin{equation}\label{growth}
	|f_i(u)| \leq K(|u|^{\mu} + 1) \quad \text{ for all } \quad i=1,\ldots, N \text{ and all } u\in \mathbb R^N.
\end{equation}

{It is frequent that system \eqref{system} possesses a set of conservation laws. Denote by $m = \mathrm{codim}\{(y_r' - y_r)_{r=1,\ldots, R}\}$. Then if $m>0$ we can find a matrix $\mathbb Q\in \mathbb R^{m\times N}$ such that $\mathbb Q f(u) = 0$ for all $u\in \mathbb R^N$. Here $f(u) = (f_1(u), \ldots, f_N(u))$. This in combination with the homogeneous Neumann boundary condition of \eqref{system} leads (formally) to $m$ conservation laws of the form
\begin{equation*}
	\mathbb Q\,\overline{u}(t) = \mathbb Q\,\overline{u}_0  \quad \text{ for all } \quad t>0.
\end{equation*}
where $\overline{u} = (\overline{u}_1,\ldots, \overline{u}_N)$ with $\overline{u}_i = \frac{1}{|\Omega|}\int_{\Omega}u_i(x)dx$ (see Section \ref{basic} for more details). 
}

\medskip
Global existence of solutions is one of the most important questions in studying reaction-diffusion systems (if not any other PDE models). Concerning system \eqref{system}-\eqref{nonlinearties}, due to the linear diffusion and polynomial-type nonlinearities, the local existence of solutions follows from standard theory of parabolic systems (see e.g. \cite{Ama85}). The global existence of solutions to \eqref{system}-\eqref{nonlinearties} on the other hand is a subtle task due to the lack of suitable {\it a priori} estimates (note that the maximum principle fails to apply to reaction-diffusion systems except very special cases). This issue has been extensively addressed in the literature, see e.g. \cite{CDF14, CV09, DFPV07, FL16, FLS16, GH97, GV10, Rot87, Tan17} and the survey \cite{Pie-Survey}. 

\medskip
Let us briefly mention the recent advances concerning the global existence of \eqref{system}-\eqref{nonlinearties}.
\begin{itemize}
	\item In \cite{GV10}, by De Giorgi's method \eqref{system} was proved to have global classical solutions in one and two dimensions with the nonlinearities of order three and two, respectively (which means $d = 1$ and $\mu = 3$ or $d = 2$ and $\mu = 2$). These results have been re-proved (and slightly improved) in \cite{Tan17} using simpler arguments. It was also proved in \cite{CV09}, again by De Giorgi's method, that if the order of nonlinearities is strictly sub-quadratic (i.e. $\mu <2$), then classical solutions exist globally in any dimension. See also \cite{GH97} for quadratic systems in heterostructure.
	\item In \cite{CDF14} and later in \cite{FLS16} the system \eqref{system} in higher dimensions and higher order of nonlinearities are shown to possess global classical solutions provided the diffusion coefficients $d_1, \ldots, d_N$ are "close to each other", i.e.
	\begin{equation*}
		\max\{d_i\} - \min\{d_i\} \text{ is small enough}.
	\end{equation*}
	\item Concerning weaker notions of solutions, it was shown in \cite{DFPV07,Pie03} that \eqref{system} with quadratic nonlinearities has global weak solutions in any dimensions, by using a duality method. An even weaker solution called "renormalized solution" was proved global in \cite{Fis15} for any dimension and any order of nonlinearities.
\end{itemize}
We remark that despite recent above mentioned advances, the global existence of classical solutions to \eqref{system} is widely open. For example, it remains unknown whether the "four-species" system modelling
\begin{equation*}
	S_1 + S_3 \leftrightharpoons S_2 + S_4
\end{equation*}
in dimension three (and higher) possesses global classical solutions or not (without assuming the closeness of the diffusion coefficients).

Recently, C\'aceres and Ca\~nizo in \cite{MC16} investigated a different regime of \eqref{system} so-called close-to-equilibrium regularity. The main idea is that when the solution stays in a small neighbourhood of an equilibrium then it is possible that the linear part is dominating the behaviour of the system, and this helps 	consequently to obtain more regularity and hence the global existence of classical solutions. The authors proved in \cite{MC16} that for dimension $d\leq 4$, under the assumption that the nonlinearities are at most quadratic and \eqref{system} satisfies a detailed balanced condition (see Definition \ref{complex-balance}), if initial data is close to an equilibrium in $L^2$-norm then the classical solutions exists globally and converge exponentially to the equilibrium in $L^{\infty}$-norm.

The aim of this paper is to extend the results in \cite{MC16} to systems with complex balance condition and with higher order of nonlinearities. 

Our main result reads as follows.
\begin{theorem}\label{theo:main}
	{
	Let $\Omega\subset \mathbb R^d$ be a bounded domain with smooth boundary $\partial\Omega$ (e.g. $C^{2+\epsilon}$ for $\epsilon>0$). 
	Assume that $d\leq 4$ and the order of nonlinearities $\mu$ (defined in \eqref{mu}) satisfies
\begin{center}
  \begin{tabular}{ |c|c|c|c|}
    \hline
   $d=1$ & $d=2$ & $d=3$ & $d=4$ \\ \hline
    $\mu= 5$ & $\mu= 3$ & $\mu= 7/3$ & $\mu= 2$\\
    \hline
  \end{tabular}
\end{center}
or equivalently $\mu = \frac{d+4}{d}$ for $d\leq 4$.

Assume that the system \eqref{system} is complex balanced, which implies for any positive initial mass there exists a unique corresponding strictly positive complex balanced equilibrium (see Definition \ref{complex-balance} and Proposition \ref{unique}). 

Now let $u_{\infty}$ be any strictly positive complex balanced equilibrium. Then there exists an $\varepsilon>0$ small enough such that for any nonnegative initial data $u_0 = (u_{i,0})\in L^{\infty}(\Omega)^N$ satisfying $\mathbb Q\,\overline{u}_0 = \mathbb Q\,u_{\infty}$, and being close to $u_{\infty} \in (0,\infty)^N$ in $L^2$-norm, i.e.
\begin{equation*}
	\sum_{i=1}^{N}\|u_{i,0} - u_{i,\infty}\|_{L^2(\Omega)} \leq \varepsilon,
\end{equation*}
the classical solution to \eqref{system} exists globally, and converges to $u_{\infty}$ in $L^{\infty}$-norm exponentially, i.e.
\begin{equation*}
	\sum_{i=1}^{N}\|u_i(t) - u_{i,\infty}\|_{L^{\infty}(\Omega)} \leq Ce^{-\lambda t} \quad \text{ for all } \quad t>0,
\end{equation*}
with $C>0$ and $\lambda>0$ are constants.}
\end{theorem}
Theorem \ref{theo:main} improves the results of \cite{MC16} in the following senses:
\begin{itemize}
	\item Firstly, it allows to treat systems with higher order nonlinearities, in particular in one and two dimensions. The main idea which leads to this improvement is to utilize the Gagliardo-Nirenberg inequality, especially in one and two dimensions, which helps to treat nonlinearities of order higher than two. Note that the idea was also used in the recent paper \cite{Tan17} to prove the global existence of classical solutions to reaction-diffusion systems in one and two dimensions (without the assumption of initial data being close to equilibrium).
	\medskip
	\item Secondly, it applies to complex balanced systems which are more general than detailed balanced systems in \cite{MC16} (see Definition \ref{complex-balance}). One of the main steps in proving Theorem \ref{theo:main} is obtaining a spectral gap for the linearised operator. This step was done in \cite{MC16} thanks to a natural Lyapunov functional of \eqref{system} inherited from the detailed balanced condition. Extending this to systems with complex balance condition requires nontrivial calculations (see Lemma \ref{spectral-gap}).
\end{itemize}

Last but not least, Theorem \ref{theo:main} successfully treats systems with boundary equilibria, i.e. equilibria belonging to $\partial\mathbb R_+^N$ (see Remark \ref{boundary-equilibrium}). It is remarked that for general complex balanced system \eqref{system} {\it without boundary equilibria}, it was proved in \cite{DFT16,FT17,Mie16} that any renoramlised solution converges exponentially in $L^1$-norm to equilibrium. The situation is quite different with the occurrence of boundary equilibria, since in such cases the methods in these forementioned works do not apply. It is however conjectured that for any complex balanced system the strictly positive equilibrium is the only attracting point. This is now named {\it Global Attractor Conjecture} and has remains as one of the most important open question in chemical reaction networks.
Theorem \eqref{theo:main} shows that despite of boundary equilibria, any solution with initial data being close enough to any strictly positive complex balanced equilibrium in $L^2$-norm, converges exponentially to that	 equilibrium. The results of this paper hence are also the extensions of local stability of complex balanced systems in ODE settings (see \cite{HJ72,SJ08}).

\medskip
Let us briefly describe the method used in this paper. First, we consider the linearised system of \eqref{system} around a strictly positive complex balanced equilibrium. By utilising the complex balanced condition, we obtain a spectral gap for the linearised system. Due to the restrictions on dimensions and order of nonlinearities, it is then shown that the linear part is dominating in the behaviour of \eqref{system} in the close-to-equilibrium regime, that means the $L^2$-norm of solution converges exponentially to the equilibrium. After that, by utilising the Gagliardo-Nirenberg inequality and smoothing effect of the heat operator, we get the global existence of classical solution whose $L^\infty$-norm grows at most polynomially in time, which in a combination with the exponential convergence in $L^2$-norm leads to the global convergence of solutions to the complex balanced equilibrium in $L^{\infty}$-norm.

\medskip
The restriction $d\leq 4$ and corresponding $\mu =  1 + \frac 4d$ is due to the fact that small $L^2$-initial data leads only to control of $L^2$-norm of solutions, which is only enough with the current techniques to control the higher norm under the mentioned assumptions on dimension $d$ and the order of nonlinearity $\mu$. The extension to arbitrary dimension seems to need more delicate analysis instead of using only the boundedness in $L^2$ (see e.g. \cite{PS00} for a system which has solutions being bounded in $L^2$ but blowing up in $L^{\infty}$!). For extending this work to higher dimension $d\geq 5$ and $\mu \geq 3$, it is naturally expected that the initial data should be close to equilibrium in $L^p$-norm for some $p>2$ depending on $d$ and $\mu$. Still the main obstacle is to obtain a global bound in time of $L^p$-norm (or $L^{p+\varepsilon}$-norm) of the solution. We discuss this extension in more details in Section \ref{extension}.

\medskip
{\bf Notations:} Throughout this paper, we denote by $\|\cdot\|_p$ the norm in $L^p(\Omega)$ for $1\leq p \leq \infty$. For any other Banach space $X$ then $\|\cdot\|_X$ is used for its norm. The inner product in $L^2(\Omega)$ is written as $\langle \cdot, \cdot \rangle$. For any $T> t_0 \ge 0$ we denote by $Q_{t_0,T} = \Omega\times (t_0,T)$ and $L^p(Q_{t_0,T}) = L^p(t_0,T;L^p(\Omega))$. When $t_0=0$ we simply write $Q_T = Q_{0,T}$.

\medskip
The rest of the paper is organised as follows: In section \ref{basic} we recall basic concepts concerning complex balanced chemical reaction networks. Section \ref{sec:spectral-gap} deals with the linearised system around a strictly positive complex balanced equilibrium and proved that the linearised system has a spectral gap. Section \ref{sec:main} is devoted for the proof of Theorem \ref{theo:main}. Finally, we discuss in the last section possible extensions to higher dimensions.

\section{Complex balanced chemical reaction networks}\label{basic}
In this section, we gather the basic concepts of chemical reaction networks with complex balanced condition. For more details, the interested reader is referred to \cite{Fei79,Fei87} for the ODE setting and also \cite{DFT16,FT17} for the PDE setting. Recalling that we consider $N$ chemical substances $S_1, S_2, \ldots, S_N$ reacting in $R$ reactions of the form
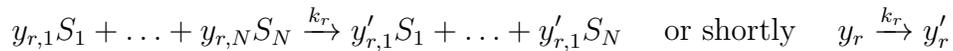
\begin{equation*}
	y_{r,1}S_1 + \ldots + y_{r,N}S_N \xrightarrow{k_r} y_{r,1}'S_1 + \ldots + y_{r,1}'S_N \quad \text{ or shortly } \quad y_r \xrightarrow{k_r} y_r'
\end{equation*}
where $y_r, y_r' \in \mathbb N^N$ are stoichiometric coefficients. Under the assumptions of no out-flux and using the law of mass action, one arrives at the reaction-diffusion system \eqref{system}. Since the system under consideration is closed, there are often conservation of masses. Indeed, denote by $W = (y_r' - y_r)_{r=1,\ldots, R} \in \mathbb R^{N\times R}$ the Wegscheider's matrix. Let $m = \mathrm{codim}(W)$. In case $m>0$, there exists a matrix $\mathbb Q\in \mathbb R^{m\times N}$ whose rows are basis of $\ker(W^\top)$. Since 
\begin{equation}
	f(u) = (f_1(u), \ldots, f_N(u)) = \sum_{r=1}^{R}k_r(y_r' - y_r)u^{y_r} \in \mathrm{range}(W)
\end{equation}
we have, thanks to the homogeneous Neumann boundary condition, formally
\begin{equation*}
	\frac{d}{dt}\int_{\Omega}\mathbb Q\,udx = \int_{\Omega}\mathbb Q\,f(u)dx = 0,
\end{equation*}
thus
\begin{equation}\label{conservation}
	\mathbb Q\,\overline{u}(t) = \mathbb Q\,\overline{u}_0 =: M \in \mathbb R^m \quad \text{ for all } t>0
\end{equation}
where $\overline{u} = (\overline{u_1}, \ldots, \overline{u_N})$ with $\overline{u_i} = \frac{1}{|\Omega|}\int_{\Omega}u_idx$, is the vector of averages.
In the case $m = 0$, then the system \eqref{system} has no conservation laws. {Note that we can change signs of some rows of $\mathbb Q$ if necessary to have $M$ positive (componentwise). Therefore, from now on, we always consider positive initial mass.} 

\medskip
There are two important classes of (chemical) equilibrium for \eqref{system}: \textit{detailed balanced equilibrium} and \textit{complex balanced equilibrium}.
\begin{definition}\label{complex-balance}
	Let $u_{\infty}\in [0,\infty)^{N}$. 
	\begin{itemize}
		\item $u_{\infty}$ is called a {\normalfont detailed balanced equilibrium} for \eqref{system} if for each forward reaction $y \xrightarrow{k_f} y'$ there exists a corresponding backward reaction $y' \xrightarrow{k_b} y$ and these two reactions are balanced at $u_{\infty}$, i.e.
		\begin{equation*}
			k_fu_{\infty}^y = k_b u_{\infty}^{y'}.
		\end{equation*}
		\item $u_{\infty}$ is called a {\normalfont complex balanced equilibrium} for \eqref{system} if at this equilibrium the total in-flow and out-flow are balanced for each complex, i.e. for any $y\in \mathcal C$ there holds
			\begin{equation}\label{CB-condition}
				\sum_{\{r:\; y_r = y\}}k_ru_{\infty}^{y_r} = \sum_{\{r:\; y_r' = y\}}k_ru_{\infty}^{y_r}.
			\end{equation}
	\end{itemize}
\end{definition}

It is straightforward that 
\begin{equation*}
	u_{\infty} \text{ is a detailed balanced equilibrium} \quad \Longrightarrow \quad u_{\infty} \text{ is a complex balanced equilibrium}
\end{equation*}
but the reverse is in general not true.

If system \eqref{system} admits a strictly positive complex balanced equilibrium, then it is called a {\it complex balanced system}. Note that the condition \eqref{CB-condition} does not give a unique complex balanced equilibrium but a manifold of equilibria instead. The uniqueness of (positive) complex balanced equilibrium is nevertheless determined via the set of conservation laws \eqref{conservation}.
\begin{proposition}[Uniqueness of positive equilibrium]\cite{Fei79,Fei87}\label{unique}
	Assume that \eqref{system} is complex balanced. Then for each positive initial mass vector $0\not= M \in \mathbb R^m_+$ in \eqref{conservation}, there exists a unique {\normalfont strictly positive} complex balanced equilibrium $u_{\infty}\in (0,\infty)^N$ to \eqref{system}.
\end{proposition}
\begin{remark}\label{boundary-equilibrium}
	The strictly positive complex balanced equilibrium is uniquely determined through the initial mass vector $M$ rather the initial data. That means it's possible to have the same complex balanced equilibrium for different initial data as long as they have the same initial mass.

	It is remarked that though \eqref{system} possesses a unique strictly positive complex balanced equilibrium, it may have possibly many so-called {\normalfont boundary equilibrium} $u_{*}$, that means $u_*\in \partial\mathbb R_+^N$ and $u_{*}$ satisfies the condition \eqref{CB-condition}. For complex balanced systems without boundary equilibria it was proved that any solution converges exponentially to the strictly positive complex balanced equilibrium, see \cite{DFT16,FT17}. If a complex balanced system possesses boundary equilibria, then the large time behaviour of solutions is in general unclear. The {\normalfont Global Attractor Conjecture} asserts that despite of boundary equilibria, the strictly positive one is still the only attracting point. This conjecture however still remains unsolved in full generality.
\end{remark}
The class of complex balanced systems is an important class in chemical reaction network theory and thus has been studied extensively in the last decades. We emphasise that most of the existing works dealt with complex balanced systems in ODE settings, while the PDE settings are much less investigated. We refer the interested reader to recent works \cite{DFT16,FT17} for studies of complex balanced systems in PDE settings.

\section{Spectral gap for the linearised operator}\label{sec:spectral-gap}
This section shows that if \eqref{system} is complex balanced then its linearised system around any strictly positive complex balanced equilibrium converges exponentially to that equilibrium. Denote by $u_{\infty}$ a strictly complex balanced equilibrium to \eqref{system}, i.e. $u_{\infty}\in (0,\infty)^N$ and $u_{\infty}$ satisfies \eqref{CB-condition}. We first write down the linearised system of \eqref{system} around $u_{\infty}$. Since $f_i(u_{\infty}) = 0$ for all $i=1,\ldots, N$ we have
\begin{equation*}
	f_i(u) = \nabla f_i(u_{\infty})\cdot (u - u_{\infty}) + \text{higher oder terms}.
\end{equation*}
Denoting $v_i = u_i - u_{i,\infty}$ and $v= (v_1, \ldots, v_N)$ we obtain the linearised system around the equilibrium $u_{\infty}$ as follow
\begin{equation}\label{linear-system}
	\begin{aligned}
		\partial_tv_i - d_i\Delta v_i &= \nabla f_i(u_{\infty})\cdot v =: L_iv, && x\in\Omega, \quad t>0,\\
		\nabla v_i \cdot \nu &= 0, && x\in\partial\Omega, \quad t>0,\\
		v_i(x,0) &= u_{i,0}(x) - u_{i,\infty}, && x\in\Omega,
	\end{aligned}
\end{equation}
with
\begin{equation}\label{linearised-operator}
	L_iv = \nabla f_i(u_{\infty})\cdot v =   \sum_{r=1}^{R}k_ru_{\infty}^{y_r}\sum_{j=1}^{N}(y_{r,i}' - y_{r,i})y_{r,j}\frac{v_j}{u_{j,\infty}}.
\end{equation}

\begin{lemma}\label{3.1}
	Assume that $u_{\infty}$ is a (strictly positive) complex balanced equilibrium. Then we have the following identity
	\begin{equation}\label{identity}
		\sum_{i=1}^{N}L_iv\frac{v_i}{u_{i,\infty}} = - \frac{1}{2}\sum_{r=1}^{R}k_ru_{\infty}^{y_r}\left(\sum_{i=1}^{N}(y_{r,i}' - y_{r,i})\frac{v_i}{u_{i,\infty}}\right)^2 \quad \text{ for all } v\in \mathbb R^N.
	\end{equation}
\end{lemma}
\begin{proof}
	We proof the identity by comparing the coefficients $\frac{v_iv_j}{u_{i,\infty}u_{j,\infty}}$ for all $i, j = 1,\ldots, N$. First of all the left hand side of \eqref{identity} is rewritten as
	\begin{equation}\label{t1}
		\begin{aligned}
		\sum_{i=1}^NL_iv\frac{v_i}{u_{i,\infty}} &= \sum_{r=1}^Rk_ru_{\infty}^{y_r}\sum_{i=1}^{N}\sum_{j=1}^{N}(y_{r,i}' - y_{r,i})y_{r,j}\frac{v_i}{u_{i,\infty}}\frac{v_j}{u_{j,\infty}}\\
		&= \sum_{r=1}^Rk_ru_{\infty}^{y_r}\sum_{i=1}^{N}\sum_{j=1}^{N}(y_{r,i}'y_{r,j} - y_{r,i}y_{r,j})\frac{v_i}{u_{i,\infty}}\frac{v_j}{u_{j,\infty}}.
		\end{aligned}
	\end{equation}
	Hence, for any $i=1,\ldots, N$, the coefficient of $v_i^2/u_{i,\infty}^2$ in the left hand side of \eqref{identity} is
	\begin{equation}\label{left-coe}
		\sum_{r=1}^{R}k_ru_{\infty}^{y_r}(y_{r,i}'y_{r,i} - y_{r,i}^2)
	\end{equation}
	and the coefficient in the right hand side is
	\begin{equation}\label{right-coe}
		\frac 12 \sum_{r=1}^{R}k_ru_{\infty}^{y_r}(2y_{r,i}y_{r,i}' - y_{r,i}^2 - y_{r,i}'^2) = \sum_{r=1}^{R}k_ru_{\infty}^{y_r}(y_{r,i}'y_{r,i} - \frac{1}{2}(y_{r,i}^2 + y_{r,i}'^2)).
	\end{equation}
	For \eqref{left-coe} and \eqref{right-coe} to be equal, we need to show
	\begin{equation}\label{t2}
		\sum_{r=1}^Rk_ru_{\infty}^{y_r}y_{r,i}^2 = \sum_{r=1}^{R}k_ru_{\infty}^{y_r}y_{r,i}'^2.
	\end{equation}
	By rearranging the sums, \eqref{t2} can be rewritten as
	\begin{equation*}
		\sum_{\{y\in \mathcal C\}} \sum_{\{r:\,y_r = y \}}k_ru_{\infty}^{y_r}y_{r,i}^2 = \sum_{\{y\in\mathcal{C}\}}\sum_{\{r:\, y_r' = y \}} k_ru_{\infty}^{y_r}y_{r,i}'^2,
	\end{equation*}
	or equivalently
	\begin{equation*}
		\sum_{\{y\in \mathcal C\}} y_i^2\sum_{\{r:\,y_r = y \}}k_ru_{\infty}^{y_r} = \sum_{\{y\in\mathcal{C}\}} y_i^2\sum_{\{r:\, y_r' = y \}} k_ru_{\infty}^{y_r}.
	\end{equation*}
	This equality follows directly from the complex balance condition \eqref{CB-condition} by multiplying both sides of \eqref{CB-condition} with $y_i^2$ then summing over $y\in \mathcal C$ . Hence, the coefficients of $v_i^2/u_{i,\infty}^2$ on both sides of \eqref{identity} are equal.

	Now for the coefficients of $v_iv_i/(u_{i,\infty}u_{j,\infty})$ we need to show, from \eqref{identity}, that
	\begin{equation*}
		\sum_{r=1}^{R}k_ru_{\infty}^{y_r}\left(y_{r,i}'y_{r,j} - y_{r,i}y_{r,j} + y_{r,j}'y_{r,i} - y_{r,j}y_{r,i}\right) = - \sum_{r=1}^R k_ru_{\infty}^{y_r}(y_{r,i}' - y_{r,i})(y_{r,j}' - y_{r,j})
	\end{equation*}
	or equivalently
	\begin{equation*}
		\sum_{r=1}^{R}k_ru_{\infty}^{y_r}y_{r,i}y_{r,j} = \sum_{r=1}^R k_ru_{\infty}^{y_r}y_{r,i}'y_{r,j}'.
	\end{equation*}
	This can be proved similarly to \eqref{t2} so we omit it here.
\end{proof}
\begin{remark}
	The proof of Lemma \ref{3.1} has to utilise the complex balanced condition \eqref{CB-condition} and hence is a nontrivial extension from the case of detailed balanced condition systems in \cite{MC16}.
\end{remark}

The next crucial lemma is the main part of this section. It shows that the linearised operator has a spectral gap. This results mostly from the identity \eqref{identity} in Lemma \ref{3.1} and the conservation laws \eqref{conservation} of \eqref{system}. Note that we have  $\mathbb Q\,\overline u(t) = M = \mathbb Q\, u_{\infty}$, thus $\mathbb Q\, \overline v(t) = 0$ for all $t>0$. 
\begin{lemma}\label{spectral-gap}
	There exists a constant $\lambda>0$ such that, for any $v\in H^1(\Omega)^N$ satisfying $\mathbb Q\,\overline v = 0$ there holds
	\begin{equation*}
		-\sum_{i=1}^{N}d_i\left\langle \Delta v_i, \frac{v_i}{u_{i,\infty}}\right\rangle - \sum_{i=1}^{N}\left\langle L_iv, \frac{v_i}{u_{i,\infty}}\right\rangle \geq \lambda \sum_{i=1}^{N}\frac{\|v_i\|_2^2}{u_{i,\infty}}
	\end{equation*}
	recalling that $\langle \cdot, \cdot \rangle$ is the inner product in $L^2(\Omega)$.
\end{lemma}
\begin{proof}
	By using Lemma \ref{3.1} and integration by parts we have 
	\begin{equation*}
		\begin{gathered}
		-\sum_{i=1}^{N}d_i\left\langle \Delta v_i, \frac{v_i}{u_{i,\infty}}\right\rangle - \sum_{i=1}^{N}\left\langle L_iv, \frac{v_i}{u_{i,\infty}}\right\rangle\\
		 = \sum_{i=1}^{N}d_i\frac{\|\nabla v_i\|_2^2}{u_{i,\infty}} + \frac 12\sum_{r=1}^{R}k_ru_{\infty}^{y_r}\int_{\Omega}\left(\sum_{i=1}^{N}(y_{r,i}' - y_{r,i})\frac{v_i}{u_{i,\infty}}\right)^2dx.
		 \end{gathered}
	\end{equation*}
	Firstly, by the Poincar\'e inequality $\|\nabla f\|_2^2 \geq P(\Omega)\|f - \overline{f}\|_2^2$, recalling that $\overline{f} = \frac{1}{|\Omega|}\int_{\Omega}f dx$ we have
	\begin{equation}\label{e1}
		\sum_{i=1}^{N}d_i\frac{\|\nabla v_i\|_2^2}{u_{i,\infty}} \geq P(\Omega)\min_{i=1,\ldots, N}\{d_i\}\sum_{i=1}^{N}\frac{\|v_i - \overline{v}_i\|_2^2}{u_{i,\infty}}.
	\end{equation}
	On the other hand, by Jensen's inequality we get,
	\begin{equation*}
		\frac 12\sum_{r=1}^{R}k_ru_{\infty}^{y_r}\int_{\Omega}\left(\sum_{i=1}^{N}(y_{r,i}' - y_{r,i})\frac{v_i}{u_{i,\infty}}\right)^2dx \geq \frac{1}{2|\Omega|}\sum_{r=1}^{R}k_ru_{\infty}^{y_r}\left(\sum_{i=1}^{N}(y_{r,i}' - y_{r,i})\frac{\overline v_i}{u_{i,\infty}}\right)^2.
	\end{equation*}
	We will now prove that there exists $\beta >0$ such that for all $\overline{v}\in \mathbb R^N$ with $\mathbb Q\,\overline{v} = 0$ there holds
	\begin{equation}\label{e2}
		\frac{1}{2|\Omega|}\sum_{r=1}^{R}k_ru_{\infty}^{y_r}\left(\sum_{i=1}^{N}(y_{r,i}' - y_{r,i})\frac{\overline v_i}{u_{i,\infty}}\right)^2 \geq \beta \sum_{i=1}^{N}\frac{\overline v_i^2}{u_{i,\infty}}.
	\end{equation}
	
We prove \eqref{e2} by a contradiction argument. Assume that 
	\begin{equation}\label{3.7.1}
		\inf_{\overline{v}\in \mathbb R^N,\; \mathbb Q\,\overline{v} = 0}\frac{\frac{1}{2|\Omega|}\sum_{r=1}^{R}k_ru_{\infty}^{y_r}\left(\sum_{i=1}^{N}(y_{r,i}' - y_{r,i})\frac{\overline v_i}{u_{i,\infty}}\right)^2}{\sum_{i=1}^{N}\frac{\overline v_i^2}{u_{i,\infty}}}  = 0.	
	\end{equation}
	
	Since both the nominator and denominator have the homogeneity of order two, we can assume w.l.o.g. that $\|\overline{v}\| = 1$ with $\|\cdot\|$ denotes the Euclidean norm in $\mathbb R^N$. Thanks to this assumption, the denominator is bounded from above. 
	If \eqref{3.7.1} holds then there exists $\overline{v} \in \mathbb R^N$ with $\mathbb Q\,\overline{v} = 0$ and $\|\overline{v}\| = 1$ such that
	\begin{equation*}
		\sum_{r=1}^{R}k_ru_{\infty}^{y_r}\left(\sum_{i=1}^N(y_{r,i}' - y_{r,i})\frac{\overline{v_i}}{u_{i,\infty}}\right)^2 = 0
	\end{equation*}
	which implies for all $r = 1, \ldots, R$,
	\begin{equation*}
		\sum_{i=1}^{N}(y_{r,i}' - y_{r,i})\frac{\overline{v_i}}{u_{i,\infty}} = 0.
	\end{equation*}
	Recalling the Wegscheider's matrix $W = (y_r' - y_r)_{r=1,\ldots, R}$, then this implies that
	\begin{equation*}
		\mathrm{diag}\left(\frac{1}{u_{1,\infty}}, \ldots, \frac{1}{u_{N,\infty}}\right)\overline{v} = \left(\frac{\overline{v_1}}{u_{1,\infty}}, \ldots, \frac{\overline{v_N}}{u_{N,\infty}}\right) \in \ker(W^{\top}).
	\end{equation*}
	If $\ker{W^{\top}} = \{0\}$ then it follows immediately $\overline{v} = 0$, contradicting with $\|\overline{v}\| = 1$. In case $\ker{W^{\top}} \not= \{0\}$, we recall that the rows of $\mathbb Q$ form a basis of $\ker(W^{\top})$ (see Section \ref{basic}), hence 
	\begin{equation*}
		\mathrm{diag}\left(\frac{1}{u_{1,\infty}}, \ldots, \frac{1}{u_{N,\infty}}\right)\overline{v} = \mathbb Q^{\top}\xi \quad \text{ for some } \xi \in \mathbb R^m
	\end{equation*}
	and thus $\overline{v} = \mathrm{diag}(u_{1,\infty}, \ldots, u_{N,\infty})\mathbb Q^{\top}\xi$. {By the assumption $\mathbb Q\,\overline{v} = 0$ we have
	\begin{equation*}
		\mathbb Q\,\mathrm{diag}(u_{1,\infty}, \ldots, u_{N,\infty})\mathbb Q^{\top}\xi = 0 \quad \text{ or } \quad \mathbb Q\,\mathcal D\,\mathbb Q^{\top}\xi = 0
	\end{equation*}
	where $\mathcal D = \mathrm{diag}(u_{1,\infty}, \ldots, u_{N,\infty})$. Now
	\begin{equation*}
		0 = (\xi, \mathbb Q\,\mathcal D\,\mathbb Q^{\top}\xi) = (\xi, \mathbb Q\,\mathcal D^{1/2}\,\mathcal{D}^{1/2}\,\mathbb Q^{\top}\xi) = (\mathcal D^{1/2}\mathbb Q^{\top}\xi, \mathcal D^{1/2}\mathbb Q^{\top}\xi),
	\end{equation*}
	in which $(\cdot,\cdot)$ denotes the inner product in $\mathbb R^m$, hence $\xi = 0$ since the matrix $\mathcal D^{1/2}\mathbb Q^{\top}$ has full rank.} But $\xi = 0$ leads to $\overline{v} = 0$ which contradicts with $\|\overline{v}\| = 1$. Therefore \eqref{e2} is proved. Finally, a combination of \eqref{e1} and \eqref{e2} gives us the desired estimate of this Lemma.
\end{proof}
\section{Proof of Theorem \ref{theo:main}}\label{sec:main}

In this section we denote again by $u_{\infty} = (u_{1,\infty}, \ldots, u_{N,\infty})$ a strictly positive complex balanced equilibrium to \eqref{system}. 

Since the system \eqref{system}--\eqref{nonlinearties} has diagonal linear diffusion matrix and the nonlinearities are polynomial, thus locally Lipschitz, the local existence of classical solution is well known from literature. By classical solution to \eqref{system} on $[0,T)$ we mean that $u\in C([0,T); L^2(\Omega)^N]\cap L^{\infty}(Q_{T-\tau})$ for all $\tau \in (0,T)$, and $\partial_tu_i$, $\partial_{x_j}u_i$, $\partial_{x_jx_k}u_i, f_i(u) \in L^p(Q_{\tau,T})$ for all $\tau\in (0,T)$, all $j,k = 1,\ldots, d$, all $p\in [1,\infty)$ and all $i = 1,\ldots, N$, and the system \eqref{system} is satisfied a.e..
\begin{theorem}[Local existence]\cite{Ama85,Pie-Survey,Rot87}\label{local}
	Let $\Omega\subset\mathbb R^d$ be a bounded domain with smooth boundary $\partial\Omega$ (e.g. $\partial\Omega$ is of class $C^{2+\epsilon}$ with $\epsilon>0$). Assume that the initial data $u_0 \in L^{\infty}(\Omega)^N$ is nonnegative. Then there exists a maximal time interval $(0,T_{max})$ such that \eqref{system} possesses a unique local nonnegative classical solution $u = (u_1,\ldots, u_N)$ to \eqref{system}--\eqref{nonlinearties} on $(0,T_{max})$. This solution satisfies in particular $u_i \in C([\tau, T_{max}); H^2(\Omega))$ for any $\tau\in (0,T_{max})$ and for all $i=1,\ldots, N$.

Moreover,
\begin{equation*}
	\lim_{t\to T_{max}-}\|u_i(t)\|_{L^{\infty}(\Omega)} < +\infty \quad \text{ for all } \quad i=1,\ldots, N \quad \Longrightarrow \quad T_{max} = +\infty.
\end{equation*}
\end{theorem}
\begin{remark}
	The non-negativity of local solution in Theorem \ref{local} follows from the fact that the nonlinearities $f_i(u)$ defined in \eqref{nonlinearties} satisfies a so-called {\it quasi-positivity} property, that is
	\begin{equation*}
		f_i(u) \geq 0 \text{ for all } u\in \mathbb R_+^N \text{ and } u_i = 0.
	\end{equation*}
	This property has a simple physical interpretation: if a chemical substance $S_i$ has zero concentration then it cannot be consumed in the corresponding reaction \eqref{reactions}.
\end{remark}
Since $u\in C([0,T_{max});L^2(\Omega)^N)$, the following corollary follows immediately.
\begin{corollary}\label{shift_initial}
	There exists $0<t_0<T_{max}$ such that 
	\begin{equation*}
		\sum_{i=1}^{N}\|u_i(t_0) - u_{i,\infty}\|_2 \leq 2\sum_{i=1}^{N}\|u_{i,0} - u_{i,\infty}\|_2.
	\end{equation*}
\end{corollary}
Thanks to this corollary we can shift the initial time to $t_0>0$ to make use of the fact that $u_i \in C([t_0,T_{max}); H^2(\Omega))$ and still keep the closeness to equilibrium $u_{\infty}$ when the initial data is close to $u_{\infty}$.

\medskip
System \eqref{system} is rewritten as
\begin{equation*} 
	\begin{aligned}
		\partial_t u_i - d_i\Delta u_i &= \nabla f_i(u_{\infty})\cdot (u - u_{\infty}) + [f_i(u) - \nabla f_i(u_{\infty})\cdot (u - u_{\infty})], && x\in\Omega,\\
		\nabla u_i \cdot \nu &= 0, && x\in\partial\Omega,\\
		u_i(x,0) &= u_{i,0}(x), && x\in\Omega.
	\end{aligned}
\end{equation*}
Denote by $v_i = u_i - u_{i,\infty}$, the system for $v = (v_1, \ldots, v_N)$ reads as
\begin{equation}\label{Taylor}
	\begin{aligned}
		\partial_t v_i - d_i\Delta v_i &= L_i v + f_i(v+u_{\infty}) - \nabla f_i(u_{\infty})\cdot v =: L_i v + g_i(v), && x\in\Omega,\\
		\nabla v_i \cdot \nu &= 0, && x\in\partial\Omega,\\
		v_i(x,0) &= u_{i,0}(x) - u_{i,\infty}, && x\in\Omega
	\end{aligned}
\end{equation}
recalling that $L_iv$ is defined in \eqref{linearised-operator}. 
\begin{lemma}\label{g}
	There exists a constant $\delta \in (0,1]$ such that for all $i=1,\ldots, N$,
	\begin{equation}\label{e2_1}
		|g_i(v)| \leq C\sum_{j=1}^N(|v_j|^{\mu} + |v_j|^{1+\delta}) \quad \text{ for all } v\in \mathbb R^N_+,
	\end{equation}
	for a constant $C>0$.
\end{lemma}
\begin{proof}
	Define $\delta \in (0,1]$ as
	\begin{equation*}
		1+\delta = \min\left\{2; \min_{r=1,\ldots, R, j=1,\ldots, N}\{y_{r,j}:  y_{r,j}>1\} \right\}.
	\end{equation*}
	We write
	\begin{equation}\label{g1}
		\begin{aligned}
			|g_i(v)| &= |f_i(u) - \nabla f_i(u_{\infty})\cdot v|\\
			&= \left|\sum_{r=1}^Rk_r(y_{r,i}' - y_{r,i})\left((v+u_{\infty})^{y_r} - \sum_{j=1}^Ny_{r,j}\frac{u_{\infty}^{y_r}}{u_{j,\infty}}v_j \right)\right|\\
			&= \left|\sum_{r=1}^Rk_r(y_{r,i}' - y_{r,i})\left((v+u_{\infty})^{y_r} - u_{\infty}^{y_r} -  \sum_{j=1}^Ny_{r,j}\frac{u_{\infty}^{y_r}}{u_{j,\infty}}v_j \right)\right|
		\end{aligned}
	\end{equation}
	where we have used $f_i(u_{\infty}) = \sum_{r=1}^{R}k_r(y_{r,i}' - y_{r,i})u_{\infty}^{y_r} = 0$ in the last step. By Taylor's expansion we have
	\begin{equation*}
		(v_j + u_{j,\infty})^{y_{r,j}} = u_{j\infty}^{y_{r,j}} + y_{r,j}u_{j,\infty}^{y_{r,j}-1}v_j + \mathbf{1}_{y_{r,j} > 1}O(|v_j|^{y_{r,j}})
	\end{equation*}
	where $\mathbf{1}_{y_{r,j}>1} = 1$ if $y_{r,j}>1$ and $\mathbf{1}_{y_{r,j}>1} = 0$ otherwise.
	Hence
	\begin{equation*}
		\begin{aligned}
			&(v+u_{\infty})^{y_r} - u_{\infty}^{y_r} -  \sum_{j=1}^Ny_{r,j}\frac{u_{\infty}^{y_r}}{u_{j,\infty}}v_j\\
			&=\prod_{j=1}^N\left[u_{j,\infty}^{y_{r,j}} + y_{r,j}u_{j,\infty}^{y_{r,j}-1}v_j + \mathbf{1}_{y_{r,j} > 1} O(|v_j|^{y_{r,j}})  \right]- u_{\infty}^{y_r} -  \sum_{j=1}^Ny_{r,j}\frac{u_{\infty}^{y_r}}{u_{j,\infty}}v_j\\
			&\leq C\sum_{j=1}^{N}|v_j|^{\mu} + C\sum_{j=1\ldots N:\;y_{r,j}>1}|v_j|^{y_{r,j}}\\
			&\leq C\sum_{j=1}^{N}(|v_j|^{\mu} + |v_j|^{1+\delta}).
		\end{aligned}
	\end{equation*}
	Inserting this into \eqref{g1} we obtain the desired estimate \eqref{e2_1}.
\end{proof}

The local existence of classical solution to \eqref{Taylor} follows from Theorem \ref{local}. Moreover, thanks to Corollary \ref{shift_initial}, there exists $t_0 > 0$ such that
\begin{equation}\label{small_v}
	\sum_{i=1}^{N}\|v_i(t_0)\|_2 \leq 2\sum_{i=1}^{N}\|v_{i,0}\|_2.
\end{equation}
\begin{lemma}\label{decay-L2}
	There exists $\varepsilon>0$ small enough such that if $\sum_{i=1}^{N}\|v_{i,0}\|_2 \leq \varepsilon$ then the solution $v = (v_1, \ldots, v_N)$ to \eqref{Taylor} satisfies
	\begin{equation*}
		\sum_{i=1}^{N}\frac{\|v_i(t)\|_2^2}{u_{i,\infty}} \leq e^{-\beta (t - t_0)}\sum_{i=1}^{N}\frac{\|v_{i}(t_0)\|_2^2}{u_{i,\infty}} \quad \text{ for all } \quad t_0 \leq t < T_{max},
	\end{equation*} 
	where $\beta>0$ is a constant independent of $t$.
\end{lemma}
\begin{proof}
Multiplying \eqref{Taylor} with $v_i/u_{i,\infty}$ and summing over $i=1,\ldots, N$, it yields with the help of \eqref{e2_1}, Lemma \ref{spectral-gap} and H\"older's inequality
\begin{equation*}
	\begin{aligned}
	\frac 12 \frac{d}{dt}\sum_{i=1}^N\frac{\|v_i\|_2^2}{u_{i,\infty}} + \frac 12 \sum_{i=1}^{N}d_i\frac{\|\nabla v_i\|^2}{u_{i,\infty}} + \frac \lambda 2 \sum_{i=1}^{N}\frac{\|v_i\|_2^2}{u_{i,\infty}}
	&\leq \int_{\Omega}\sum_{i=1}^{N}\left|\frac{v_i}{u_{i,\infty}}\right||g_i(v)|dx\\
	&\leq C\int_{\Omega}\sum_{i=1}^{N}\left|\frac{v_i}{u_{i,\infty}}\right|\sum_{j=1}^{N}(|v_j|^{\mu} + |v_j|^{1+{\delta}})dx\\
	&\leq C\sum_{i=1}^{N}(\|v_i\|_{\mu+1}^{\mu+1} + \|v_i\|_{2+{\delta}}^{2+{\delta}}),
	\end{aligned}
\end{equation*}
where $\delta$ is defined in Lemma \ref{g}. Therefore for some $\beta >0$
\begin{equation}\label{e3}
	\frac 12 \frac{d}{dt}\sum_{i=1}^N\frac{\|v_i\|_2^2}{u_{i,\infty}} + \frac \beta 2 \sum_{i=1}^{N}\frac{\|v_i\|_2^2}{u_{i,\infty}}+ \frac \beta 2 \sum_{i=1}^{N}\|v_i\|_{H^1(\Omega)}^2 \leq C\sum_{i=1}^{N}(\|v_i\|_{\mu+1}^{\mu+1} + \|v_i\|_{2+{\delta}}^{2+{\delta}}).
\end{equation}
Since $d\leq 4$ and $\delta \in (0,1]$, the Sobolev embedding yields $H^1(\Omega) \hookrightarrow L^4(\Omega)\hookrightarrow L^{2+2\delta}(\Omega)$. By using the interpolation inequality we get
\begin{equation*}
	\|v_i\|_{2+\delta}^{2+\delta}\leq \|v_i\|_{2+2\delta}^2\|v_i\|_{1+\delta}^{\delta} \leq C\|v_i\|_4^2\|v_i\|_2^\delta \leq C\|v_i\|_{H^1(\Omega)}^2\|v_i\|_2^{\delta}
\end{equation*}
where we used $\|v_i\|_{2\delta} \leq C\|v_i\|_2$ in the second inequality, since $\delta \in (0,1]$.
For the term $\|v_i\|_{\mu+1}^{\mu+1}$ we apply the Gagliardo-Nirenberg inequality and recall that $\mu = \frac{d+4}{d}$ to obtain
\begin{equation*}
	\|v_i\|_{\mu+1}^{\mu+1} = \|v_i\|_{2(d+2)/d}^{2(d+2)/d} \leq C\|v_i\|_{H^1(\Omega)}^2\|v_i\|_2^{4/d}.
\end{equation*}
Therefore, it follows from \eqref{e3} that
\begin{equation*}
	\frac 12 \frac{d}{dt}\sum_{i=1}^{N}\frac{\|v_i\|_2^2}{u_{i,\infty}} + \frac \beta 2 \sum_{i=1}^{N}\frac{\|v_i\|_2^2}{u_{i,\infty}} + \frac \beta 2 \sum_{i=1}^{N}\|v_i\|_{H^1(\Omega)}^2 \leq C\sum_{i=1}^{N}\|v_i\|_{H^1(\Omega)}^2(\|v_i\|_2^{{\delta}} + \|v_i\|_2^{4/d})
\end{equation*}
which implies
\begin{equation*}
\frac{d}{dt}\sum_{i=1}^{N}\frac{\|v_i\|_2^2}{u_{i,\infty}}  \leq - \beta \sum_{i=1}^{N}\frac{\|v_i\|_2^2}{u_{i,\infty}} +  C\sum_{i=1}^{N}\|v_i\|_{H^1(\Omega)}^2(\|v_i\|_2^{\textcolor{blue}{\delta}} + \|v_i\|_2^{4/d} - C) \quad \forall t_0 \leq t < T_{max}.
\end{equation*}
{
From the assumption $\sum_{i=1}^{N}\|v_{i,0}\|_2 \leq \varepsilon$ and \eqref{small_v} we have $\sum_{i=1}^{N}\|v_{i}(t_0)\|_2 \leq 2\varepsilon$. 
Then if we choose $\varepsilon$ small enough then $\sum_{i=1}^{N}\|v_i\|_2^2/u_{i,\infty}$ decreases in time, which in fact leads to}
\begin{equation*}
	\frac{d}{dt}\sum_{i=1}^{N}\frac{\|v_i\|_2^2}{u_{i,\infty}} \leq - \beta \sum_{i=1}^{N}\frac{\|v_i\|_2^2}{u_{i,\infty}} \quad \text{ for all } t_0 \leq t < T_{max}
\end{equation*}
and consequently, by a Gronwall's lemma,
\begin{equation*}
	\sum_{i=1}^{N}\frac{\|v_i(t)\|_2^2}{u_{\infty}} \leq e^{-\beta(t - t_0)}\sum_{i=1}^{N}\frac{\|v_{i}(t_0)\|_2^2}{u_{i,\infty}} \quad \text{ for all } t_0 \leq t < T_{max}.
\end{equation*}
\end{proof}

\noindent{\bf Important notation:} From now on we always denote by $C_T$ a positive constant which depends {\it at most polynomially on $T$}, that is there exists a polynomial $P(x)$ such that $C_T \leq P(T)$ for all $T>0$.

\medskip
The following two lemmas are important in getting global classical solutions to \eqref{Taylor}.
\begin{lemma}\label{heat-regularity}\cite[Lemma 3.3]{CDF14}
	Let $f\in L^p(Q_T)$ with $p\geq 1$ satisfy $\|f\|_{L^p(Q_T)} \leq C_T$ and $u$ be the solution to the heat equation
	\begin{equation}\label{heat-equation}
		\begin{aligned}
			y_t - d\Delta y &= f, && x\in \Omega,\\
			\nabla y\cdot \nu &= 0, && x\in \partial\Omega,\\
			y(x,0) &= y_0(x), && x\in\Omega
		\end{aligned}
	\end{equation}
	with $y_0 \in L^{\infty}(\Omega)$. 
	\begin{itemize}
		\item[(i)] If $p < (d+2)/2$ then 
		\begin{equation}\label{tron}
		\|y\|_{L^s(Q_T)} \leq C_T \quad \text{ for all } \quad s < \frac{(d+2)p}{d+2-2p}.
		\end{equation}
		\item[(ii)] If $p= (d+2)/2$ then 
		\begin{equation}\label{vuong}
		\|y\|_{L^r(Q_T)} \leq C_T \quad \text{ for all } \quad 1\leq r <+\infty.
		\end{equation}
		\item[(iii)] If $p> (d+2)/2$ then
		\begin{equation}\label{meo}
			\|y\|_{L^{\infty}(Q_T)} \leq C_T
		\end{equation}
	\end{itemize}
\end{lemma}
\begin{proof}
	The proof of (i) and (ii) are given in \cite[Lemma 3.3]{CDF14}. We give here the proof for \eqref{meo}. Denote by $S(t):= e^{t(d\Delta)}$ the semigroup generated by the operator $d\Delta$ with Neumann boundary condition. Then the solution to \eqref{heat-equation} is represented as
	\begin{equation*}
		y(t) = S(t)y_0 + \int_0^tS(t-s)f(s)ds.
	\end{equation*}
	By the classical estimate $\|S(t)y_0\|_{r} \leq C\|y_0\|_q\left(1+ t^{-\frac{d}{2}\left(\frac 1q - \frac 1r\right)} \right)$ for all $q \leq r \leq \infty$ (see e.g. \cite[Lemma 2.5]{MC16}), we estimate for all $0 < t \leq T$,
	\begin{equation*}
		\begin{aligned}
			\|y(t)\|_{\infty} &\leq \|S(t)y_0\|_{\infty} + \int_0^t\|S(t-s)f(s)\|_{\infty}ds\\
			&\leq C\|y_0\|_{\infty} + C\int_0^t\|f(s)\|_p\left(1 + (t-s)^{-\frac{d}{2p}}\right)ds\\
			&\leq C\|y_0\|_{\infty} + C\|f\|_{L^p(Q_T)}\left[\int_0^t\left(1+(t-s)^{-\frac{d}{2(p-1)}}\right)ds \right]^{\frac{p-1}{p}}
		\end{aligned}
	\end{equation*}
	With $p > (d+2)/2$ with have $d/(2(p-1)) < 1$, and thus the integral on the right hand side converges, and has polynomial behaviour in $t$, and consequently in $T$. This completes the proof.
	\end{proof}
\begin{lemma}\label{heat-regularity-2}\cite{MC16}
	Consider the heat equation \eqref{heat-equation}. If $f\in L^{\infty}(0,T;L^p(\Omega))$ with $p>d/2$ and $\|f\|_{L^{\infty}(0,T;L^p(\Omega))} \leq C_T$, then 
	\begin{equation}\label{thoi}
		\|y\|_{L^{\infty}(Q_T)} \leq C_T
	\end{equation}
	with $C_T$ is a constant depending {\normalfont at most polynomially} on $T$.
\end{lemma}

\begin{remark}
	We remark that both regularity results in Lemmas \ref{heat-regularity} and \ref{heat-regularity-2} are classical and can be found in the literature, for instance in the classical book \cite[Corollary of Theorem 9.1]{Ladyzhenskaya}. Here we emphasise in \eqref{tron}, \eqref{vuong} and \eqref{thoi} the polynomial dependence of the constant $C_T$ on $T$ since it's important in proving Theorem \ref{theo:main}. In fact, the exponential decay in $L^{\infty}$-norm is a result of the exponential decay in lower norm, say $L^2$-norm in Lemma \ref{decay-L2}, and the polynomial growth in $T$ of higher order norms (see e.g. \eqref{interpolation}).
\end{remark}
We are now ready to give
\begin{proof}[Proof of Theorem \ref{theo:main}]
	Multiplying \eqref{Taylor} with $p|v_i|^{p-2}v_i$ (or more precisely a smoothed approximation of $p|v_i|^{p-2}v_i$ and letting the smoothing go to zero) then integrating over $\Omega$, we have
	\begin{equation}\label{e4}
		\frac{d}{dt}\|v_i\|_p^p + d_ip(p-1)\int_{\Omega}|v_i|^{p-2}|\nabla v_i|^2dx = p\int_{\Omega}\left(L_i v + g_i(v)\right)|v_i|^{p-2}v_idx.
	\end{equation}
	For the second term we have
	\begin{equation*}
		d_ip(p-1)\int_{\Omega}|v_i|^{p-2}|\nabla v_i|^2dx = \frac{4d_i(p-1)}{p}\|\nabla (v_i^{p/2})\|_2^2.
	\end{equation*}
	For the right hand side of \eqref{e4}, since $L_iv$ is linear (see \eqref{linearised-operator}) we estimate by H\"older's inequality
	\begin{equation*}
		\int_{\Omega}L_i v|v_i|^{p-2}v_idx \leq \int_{\Omega}|L_iv||v_i|^{p-1}dx \leq C\sum_{i=1}^{N}\|v_i\|_{p}^p.
	\end{equation*}
	For the term concerning the nonlinearities we use \eqref{e2_1} and H\"older's inequality to obtain
	\begin{equation*}
		\left|\int_{\Omega}g_i(v)|v_i|^{p-2}v_i dx\right| \leq C\int_{\Omega}|v_i|^{p-1}\sum_{j=1}^{N}(|v_j|^{\mu} + |v_j|^{1+{\delta}})dx \leq C\sum_{i=1}^{N}\left(\|v_i\|_{\mu+p-1}^{\mu+p-1} + \|v_i\|_{p+{\delta}}^{p+{\delta}}\right).
	\end{equation*}
	Therefore we obtain from summing \eqref{e4} over $i=1,\ldots, N$ that
	\begin{equation}\label{e5}
		\frac{d}{dt}\sum_{i=1}^{N}\|v_i\|_p^p + C\sum_{i=1}^{N}\|\nabla(v_i^{p/2})\|_2^2 \leq C\sum_{i=1}^{N}\left(\|v_i\|_{\mu+p-1}^{\mu+p-1} + \|v_i\|_{p+{\delta}}^{p+{\delta}} + \|v_i\|_p^p\right).
	\end{equation}
	Since $\Omega$ is bounded and $\mu \geq 2$ we have
	\begin{equation*}
		\|v_i\|_p^p \leq \|v_i\|_{\mu+p-1}^p \leq C(\|v_i\|_{\mu+p-1}^{\mu+p-1} + 1) \quad \text{ and } \quad \|v_i\|_{p+{\delta}}^{p+{\delta}} \leq C(\|v_i\|_{\mu+p-1}^{\mu+p-1}+1).
	\end{equation*}
	Then by adding $\sum_{i=1}^{N}\|v_i\|_p^p = \sum_{i=1}^{N}\|v_i^{p/2}\|_2^2$ to both sides of \eqref{e5} we obtain
	\begin{equation}\label{e6}
		\frac{d}{dt}\sum_{i=1}^{N}\|v_i\|_p^p + C\sum_{i=1}^{N}\|v_i^{p/2}\|_{H^1(\Omega)}^2 \leq C\sum_{i=1}^{N}\|v_i\|_{\mu+p-1}^{\mu+p-1} + C.
	\end{equation}
	We now consider the cases three $d=1$, $d=2$, $d=3$ and $d=4$ separately due to the different Sobolev embeddings in each case.
	
	\medskip
	{\bf The case $d = 1$}, which implies $\mu = 5$. Choosing $p = 4$ in \eqref{e6} we have
	\begin{equation*}
		\frac{d}{dt}\sum_{i=1}^{N}\|v_i\|_4^4 + C\sum_{i=1}^{N}\|v_i^2\|_{H^1(\Omega)}^2 \leq C\sum_{i=1}^{N}\|v_i\|_{8}^{8} + C = C\sum_{i=1}^{N}\|v_i^2\|_4^4 + C.
	\end{equation*}
	Applying the Gagliardo-Nirenberg inequality in one dimension 
	\[
		\|v_i^2\|_4^4 \leq C\|v_i^2\|_{H^1(\Omega)}^2\|v_i^2\|_{1}^2 = C\|v_i^2\|_{H^1(\Omega)}^2\|v_i\|_2^4.
	\]
	Thus
	\begin{equation}\label{plus}
		\frac{d}{dt}\sum_{i=1}^{N}\|v_i\|_4^4 + C\sum_{i=1}^{4}\|v_i^2\|_{H^1(\Omega)}^2 \leq C\sum_{i=1}^{N}\|v_i^2\|_{H^1(\Omega)}^2(\|v_i\|_2^2 - 1) + C.
	\end{equation}
	Thanks to Lemma \ref{decay-L2}, if $\sum_{i=1}^{N}\|v_{i,0}\|_2 \leq \varepsilon$ for small enough $\varepsilon$ we also have $\sum_{i=1}^{N}\|v_i(t)\|_2 \leq 2\varepsilon$ for all $t_0 \leq t < T_{max}$, and hence it follows from \eqref{plus} and the one dimensional embedding $H^1(\Omega)\hookrightarrow L^{\infty}(\Omega)$ that
	\begin{equation*}
		\frac{d}{dt}\sum_{i=1}^{N}\|v_i\|_4^4 + C\sum_{i=1}^{N}\|v_i\|_{\infty}^4 \leq C \quad \text{ for all } \quad t_0 \leq t < T_{max}.
	\end{equation*}
	Integrating on $(0,T)$ for any $T\in [t_0,T_{max})$ yields
	\begin{equation*}
		\sum_{i=1}^{N}\|v_i(T)\|_4^4 + C\sum_{i=1}^{N}\int_{t_0}^T\|v_i(\tau)\|_{\infty}^4d\tau \leq \sum_{i=1}^{N}\|v_{i}(t_0)\|_4^4 + CT.
	\end{equation*}
	That means for all $i=1, \ldots, N$ it holds $v_i \in L^{\infty}(t_0,T;L^4(\Omega))\cap L^{4}(t_0,T;L^{\infty}(\Omega))$. It then follows from interpolation that $v_i \in L^{8}(Q_{t_0,T})$ for all $i=1,\ldots, N$. This implies that the right hand side of the equation in \eqref{Taylor} belongs to $L^{8/5}(Q_{t_0,T})$, which with the help of Lemma \ref{heat-regularity}(iii) leads to 
	$\|v_i\|_{L^{\infty}}(Q_{t_0,T}) \leq C_T$. Since $C_T$ grows at most polynomially in $T$, this implies
	\begin{equation*}
		\lim_{t\to T_{max}-}\|v_i(t)\|_{\infty} < +\infty
	\end{equation*}
	hence $\lim_{t\to T_{max}-}\|u_i(t)\|_{\infty} < +\infty$ and therefore $T_{max} = +\infty$.
	
	To prove the convergence to equilibrium of $u_i$ to $u_{i,\infty}$ in $L^{\infty}$-norm, we first show that $\|v_i(T)\|_{H^2(\Omega)} \leq C_T$ for all $T \geq t_0$. Indeed, from $\|\partial_tv_i\|_{L^2(Q_{t_0,T})} \leq C_T$ and $\|v_i\|_{L^{\infty}(Q_{t_0,T})} \leq C_T$ and the growth condition of $f_i(u)$ we have
	\begin{equation*}
		\int_{t_0}^T\int_{\Omega}|f_i(v + u_{\infty})|^2dx \leq C(\|v\|_{L^{\infty}(Q_{t_0,T})}^{2\mu} + T) \leq C_T
	\end{equation*}
	and
	\begin{equation*}
		\begin{aligned}
		\int_{t_0}^T\int_{\Omega}|\partial_t(f_i(v + u_{\infty}))|^2dxdt &\leq \int_{t_0}^T\int_{\Omega}|\nabla f_i(v + u_{\infty})|^2|\partial_tv_i|^2dxdt\\
		&\leq C(\|v\|_{L^{\infty}(Q_{t_0,T})}^{2\mu - 2} + T)\|\partial_tv_i\|_{L^2(Q_{t_0,T})}^2 \leq C_T,
		\end{aligned}
	\end{equation*}
	thus
	\begin{equation*}
		\|f_i(v + u_{\infty})\|_{H^1(t_0,T;L^2(\Omega))} \leq C_T.
	\end{equation*}
	Therefore, we can apply regularizing effect of linear parabolic equation $\partial_tv_i - d_i\Delta v_i = f_i(v + u_{\infty})$ (see e.g. \cite[Theorem 7.1.5]{Evans} for Dirichlet boundary problems. Analog for Neumann boundary problem should follow similarly) to obtain
	\begin{equation*}
		\|v_i(T)\|_{H^2(\Omega)} \leq C(\|g_i(v)\|_{H^1(t_0,T;L^2(\Omega))} + \|v_i(t_0)\|_{H^2(\Omega)}) \leq C_T \quad \text{for all } T\geq t_0.
	\end{equation*}
	Therefore, by interpolation 
	\begin{equation}\label{interpolation}
		\|v_i(T)\|_{\infty} \leq C\|v_i(T)\|_{H^2(\Omega)}^{\theta}\|v_i(T)\|_2^{1 - \theta} \leq C_T^{\theta}e^{-\beta(1-\theta) T} \leq Ce^{-\beta' T} \quad \forall T\geq t_0,
	\end{equation}
	for some $0<\beta' < \beta(1-\theta)$. This in combination with $v_i \in L^{\infty}((0,t_0)\times \Omega)$ (implied from Theorem 4.1) completes the proof of Theorem \ref{theo:main} in one dimension.	
	
	\medskip
	{\bf The case $d=2$} with $\mu = 3$. Choosing again $p=4$ in \eqref{e6} and applying the two dimensional Gagliardo-Nirenberg inequality
	\begin{equation*}
		\|v_i\|_6^6 = \|v_i^2\|_3^3 \leq C\|v_i^2\|_{H^1(\Omega)}^2\|v_i^2\|_{1} = C\|v_i^2\|_{H^1(\Omega)}^2\|v_i\|_2^2
	\end{equation*}
	we get from \eqref{e6} that
	\begin{equation*}
		\frac{d}{dt}\sum_{i=1}^{N}\|v_i\|_4^4 + C\sum_{i=1}^{N}\|v_i^2\|_{H^1(\Omega)}^2 \leq C\sum_{i=1}^{N}\|v_i^2\|_{H^1(\Omega)}^2\|v_i\|_2^2 + C.
	\end{equation*}
	By using $\sum_{i=1}^{N}\|v_i\|_2$ is small enough, and the two dimensional embedding $H^1(\Omega)\hookrightarrow L^q(\Omega)$ for any $q\in [1,\infty)$ we obtain
	\begin{equation*}
		\frac{d}{dt}\sum_{i=1}^{N}\|v_i\|_4^4 + C\sum_{i=1}^{N}\|v_i\|_{2q}^4 \leq C
	\end{equation*}
	and consequently, by integrating on $(0,T)$,
	\begin{equation*}
		\sum_{i=1}^{N}\|v_i(T)\|_4^4 + C\sum_{i=1}^{N}\int_{t_0}^T\|v_i(\tau)\|_{2q}^4d\tau \leq \sum_{i=1}^{N}\|v_{i}(t_0)\|_4^4 + CT.
	\end{equation*}
	Therefore $v_i \in L^{\infty}(t_0,T;L^4(\Omega))\cap L^4(t_0,T;L^{2q}(\Omega)) \hookrightarrow L^{8 - 8/q}(Q_{t_0,T})$ by interpolation. Since $q \in [1,\infty)$ arbitrary, it holds in fact $v_i \in L^r(Q_{t_0,T})$ for all $1\leq r < 8$. Because $\mu = 3$, the right hand side of \eqref{Taylor} belongs to $L^s(Q_{t_0,T})$ for all $1\leq s < 8/3$, and in particular in $L^{2+\alpha}(Q_{t_0,T})$ with some $\alpha >0$. Lemma \ref{heat-regularity}(iii) implies $v_i\in L^{\infty}(Q_{t_0,T})$. An interpolation argument similar to \eqref{interpolation} concludes the proof of Theorem \ref{theo:main} in two dimensions.
	
		\medskip
		{\bf The case $d=3$} and $\mu=7/3$. In this case we choose $p>2$ arbitrary. Using the dimensional Sobolev embedding $H^1(\Omega)\hookrightarrow L^6(\Omega)$ in \eqref{e6} we have
		\begin{equation*}
			\frac{d}{dt}\sum_{i=1}^{N}\|v_i\|_p^p + C\sum_{i=1}^{N}\|v_i\|_{3p}^p \leq C\sum_{i=1}^{N}\|v_i\|_{p+4/3}^{p+4/3} + C.
		\end{equation*}
		By the interpolation inequality 
		\[
			\|v_i\|_{p+4/3}^{p+4/3} \leq \|v_i\|_{3p}^p\|v_i\|_2^{4/3}
		\]
		and $\sum_{i=1}^{N}\|v_i\|_2$ is small, we arrive at
		\begin{equation*}
			\frac{d}{dt}\sum_{i=1}^{N}\|v_i\|_p^p \leq C
		\end{equation*}
	and consequently
	\begin{equation*}
		\sum_{i=1}^{N}\|v_i(T)\|_p^p \leq \sum_{i=1}^{N}\|v_{i}(t_0)\|_p^p + CT,
	\end{equation*}
	which means $v_i\in L^{\infty}(t_0,T;L^p(\Omega))$ for any $p\in [1,\infty)$, and hence the right hand side of \eqref{e6} belongs to $L^{\infty}(t_0,T;L^p(\Omega))$ for any $p\in [1,\infty)$. This is enough to apply Lemma \ref{heat-regularity-2} to obtain $v_i \in L^{\infty}(Q_{t_0,T})$ with $\|v_i\|_{L^{\infty}(Q_{t_0,T})}\leq C_T$ and thus concludes Theorem \ref{theo:main} in three dimensions thanks to an argument similarly to \eqref{interpolation}.
		
	\medskip
	{\bf The case $d= 4$} and $\mu=2$. In this case we also choose $p>2$ arbitrary. Using the Sobolev embedding $H^1(\Omega)\hookrightarrow L^4(\Omega)$ in four dimensions to \eqref{e6} we have
	\begin{equation*}
		\frac{d}{dt}\sum_{i=1}^{N}\|v_i\|_p^p + C\sum_{i=1}^{N}\|v_i\|_{2p}^p \leq C\sum_{i=1}^{N}\|v_i\|_{p+1}^{p+1} + C.
	\end{equation*}
	By the interpolation inequality 
	\[
		\|v_i\|_{p+1}^{p+1} \leq \|v_i\|_{2p}^p\|v_i\|_2
	\]
	and $\sum_{i=1}^{N}\|v_i\|_2$ is small, we arrive at
	\begin{equation*}
		\frac{d}{dt}\sum_{i=1}^{N}\|v_i\|_p^p \leq C
	\end{equation*}
	and consequently
	\begin{equation*}
		\sum_{i=1}^{N}\|v_i(T)\|_p^p \leq \sum_{i=1}^{N}\|v_{i}(t_0)\|_p^p + CT,
	\end{equation*}
	which means $v_i\in L^{\infty}(t_0,T;L^p(\Omega))$ for any $p\in [1,\infty)$, and hence the right hand side of \eqref{e6} belongs to $L^{\infty}(t_0,T;L^p(\Omega))$ for any $p\in [1,\infty)$. This is enough to apply Lemma \ref{heat-regularity-2} to obtain $v_i \in L^{\infty}(Q_{t_0,T})$ with $\|v_i\|_{L^{\infty}(Q_{t_0,T})}\leq C_T$.
	
	Note that in the case of four dimensions the interpolation \eqref{interpolation} needs to be slightly modified since the embedding $H^2(\Omega) \hookrightarrow L^{\infty}(\Omega)$ does not hold. Due to the smoothing effect of the heat kernel, we know that the local solution in Theorem \ref{local} in fact belongs to $H^k(\Omega)$ for all $k\geq 1$ for any positive time, provided the domain $\Omega$ is correspondingly regular enough. Therefore, we can assume now that $u_i(t_0)\in H^3(\Omega)$ and repeat the arguments similarly to the case $d=1$ to obtain an a priori estimate $\|v_i(t)\|_{H^3(\Omega)} \leq C_T$. Hence the interpolation \eqref{interpolation}, now with the norm $\|\cdot\|_{H^3(\Omega)}$ in place of $\|\cdot\|_{H^2(\Omega)}$, leads to the exponential convergence to equilibrium in $L^{\infty}$-norm.
	\end{proof}
\section{Extension to higher dimensions}\label{extension}
As mentioned in the introduction, the main advantage of the $L^2$-setting is the spectral gap of the linearised operator which leads to a global bound on $L^2$-norm of solutions. This however also leads to the restriction $d\leq 4$ and $\mu = 1 + \frac 4d$ due to the fact that the bound of $L^2$-norm is in general not enough to control the nonlinearities of higher order and in higher dimensions. It is natural to expect that in such cases the initial data need to be small in $L^p$-norm, with some $p>2$ depending on $d$ and $\mu$. More importantly, one needs to prove the control of solution in $L^p$-norm (or $L^{p+\varepsilon}$-norm) to obtain the regularity of solution up to $L^{\infty}$. This remains as an open question. In this section we show how to proceed once the suitable control on solutions is established.
\begin{proposition}\label{pro:extension}
	Let $d\geq 3$ and $\Omega\subset \mathbb R^d$ be a bounded domain with smooth boundary $\partial\Omega$ (e.g. $\partial\Omega$ is of class $C^{2+\epsilon}$ for $\epsilon>0$). Assume that $u_0\in L^{\infty}(\Omega)^N$ and let
	\begin{equation}\label{p0}
		p_0 = \frac{d(\mu-1)}{2}.
	\end{equation}
	If the solution $v$ to the linearised system \eqref{Taylor} satisfies either
	\begin{equation}\label{assump-1}
		\sum_{i=1}^{N}\|v_i(t)\|_{p_0}  \leq \varepsilon \quad \text{ for all } \quad t>0,
	\end{equation}
	for small $\varepsilon$ or
	\begin{equation}\label{assump-2}
		\sum_{i=1}^{N}\|v_i(t)\|_{p_0+\delta} \leq \widehat C(t) \quad \text{ for all } \quad t>0,
	\end{equation}
	for some positive constant $\delta>0$ and $\widehat C(t)$ is a continuous function, then the classical solution $v$ to \eqref{Taylor}, and consequently the classical solution $u$ to \eqref{system}, exists globally.
\end{proposition}
\begin{remark}
	In the case, for instance, $d=4$ and $\mu=2$, which implies $p_0 = 2$, the assumption \eqref{assump-1} in fact can be proved thanks to a spectral gap of the linearised operator (see Lemma \ref{decay-L2}).
\end{remark}
\begin{proof}
	Let $p\geq p_0$ arbitrary. Multiplying \eqref{Taylor} by $p|v_i|^{p-2}v_i$ and proceed as the proof of Theorem \ref{theo:main} we obtain \eqref{e6}, which we rewrite here 
	\begin{equation*}
		\frac{d}{dt}\sum_{i=1}^{N}\|v_i\|_p^p + C\sum_{i=1}^{N}\|v_i^{p/2}\|_{H^1(\Omega)}^2 \leq C\sum_{i=1}^{N}\|v_i\|_{\mu+p-1}^{\mu+p-1} + C.
	\end{equation*}
	By the Sobolev embedding $H^1(\Omega) \hookrightarrow L^{2d/(d-2)}(\Omega)$, recalling that $d\geq 3$, we have
	\begin{equation*}
		\|v_i^{p/2}\|_{H^1(\Omega)}^2 \geq C\|v_i^{p/2}\|_{2d/(d-2)}^2 = C\|v_i\|_{dp/(d-2)}^p =: C\|v_i\|_q^p
	\end{equation*}
	with $q = \frac{dp}{p-2}$. Therefore we obtain
	\begin{equation}\label{e7}
		\frac{d}{dt}\sum_{i=1}^{N}\|v_i\|_p^p + C\sum_{i=1}^{N}\|v_i\|_{q}^p \leq C\sum_{i=1}^{N}\|v_i\|_{\mu+p-1}^{\mu+p-1} + C.		
	\end{equation}
	Note that $p \geq p_0 = \frac{d(\mu-1)}{2}$ implies $q = \frac{dp}{d-2} \geq \mu+p-1$ and $\mu+p-1 \geq p_0$ since $\mu \geq 1$. 
	
	\medskip
	We consider first the case when the assumption \eqref{assump-1} is satisfied, we use the interpolation inequality to have, recalling that $p_0 \leq \mu + p - 1 \leq q$,
	\begin{equation}\label{interpolation-1}
		\|v_i\|_{\mu+p-1}^{\mu+p-1} \leq \|v_i\|_q^{\theta(\mu+p-1)}\|v_i\|_{p_0}^{(1-\theta)(\mu+p-1)}
	\end{equation}
	with $\theta\in (0,1)$ satisfies
	\begin{equation*}
		\frac{1}{\mu+p-1} = \frac{\theta}{q} + \frac{1-\theta}{p_0}\quad \Longrightarrow\quad \theta = \frac{q[(\mu+p-1) - p_0]}{(q-p_0)(\mu+p-1)}.
	\end{equation*}
	From $q = dp/(d-2)$ and $p_0 = \frac{d(\mu-1)}{2}$ we can easily verify that
	\begin{equation*}
		\theta(\mu+p-1) = \frac{q[(\mu+p-1) - p_0]}{q-p_0} = p.
	\end{equation*}
	Therefore
	\begin{equation*}
		\|v_i\|_{\mu+p-1}^{\mu+p-1} \leq \|v_i\|_q^{p}\|v_i\|_{p_0}^{(1-\theta)(\mu+p-1)}
	\end{equation*}
	Inserting this into \eqref{e7} we obtain
	\begin{equation*}
		\frac{d}{dt}\sum_{i=1}^{N}\|v_i\|_p^p + C\sum_{i=1}^{N}\|v_i\|_q^p \leq C\sum_{i=1}^{N}\|v_i\|_q^p\|v_i\|_{p_0}^{(1-\theta)(\mu+p-1)} + C.
	\end{equation*}
	Thanks to \eqref{assump-1}, with $\varepsilon$ is small enough we obtain
	\begin{equation*}
		\frac{d}{dt}\sum_{i=1}^{N}\|v_i\|_p^p \leq C
	\end{equation*}
	or consequently
	\begin{equation*}
		\sum_{i=1}^{N}\|v_i(T)\|_p^p \leq \sum_{i=1}^{N}\|v_{i,0}\|_p^p + CT
	\end{equation*}
	for arbitrary $p\geq p_0$. By applying Lemma \ref{heat-regularity-2} we obtain the global existence of classical solutions. 
	
	\medskip
	In the case when \eqref{assump-2} holds, we will modify the interpolation inequality \eqref{interpolation-1} as follow
	\begin{equation*}
		\|v_i\|_{\mu+p-1}^{\mu+p-1}\leq \|v_i\|_{q}^{\theta(\mu+p-1)}\|v_i\|_{p_0+\delta}^{(1-\theta)(\mu+p-1)}
	\end{equation*}
	with $\theta \in (0,1)$ solving
	\begin{equation*}
		\frac{1}{\mu+p-1} = \frac{\theta}{q} + \frac{1-\theta}{p_0+\delta} \quad \Longrightarrow \quad  \theta = \frac{q[(\mu+p-1) - p_0 - \delta]}{(q-p_0 - \delta)(\mu+p-1)}.
	\end{equation*}
	Using again $q = dp/(d-2)$ and $p_0 = d(\mu-1)/2$ we can check that
	\begin{equation*}
		\theta(\mu+p-1) = \frac{q[(\mu+p-1) - p_0 - \delta]}{q - p_0 - \delta} < p.
	\end{equation*}
	Therefore, an application of Young's inequality leads to 
	\begin{equation*}
		\|v_i\|_{\mu+p-1}^{\mu+p-1}\leq \|v_i\|_{q}^{\theta(\mu+p-1)}\|v_i\|_{p_0+\delta}^{(1-\theta)(\mu+p-1)} \leq \kappa \|v_i\|_{q}^{p} + C_{\kappa}\|v_i\|_{p_0+\delta}^{(1-\theta)[\theta(\mu+p-1) - p]/\theta}
	\end{equation*}	
	in which $\kappa>0$ can be chosen as small as possible. Inserting this into \eqref{e7} and choosing $\kappa$ small enough we arrive at
	\begin{equation*}
		\frac{d}{dt}\sum_{i=1}^{N}\|v_i\|_p^p \leq C\sum_{i=1}^{N}\|v_i\|_{p_0+\delta}^{(1-\theta)[\theta(\mu+p-1)-p]/\theta} + C \leq C\widehat C(t)^{(1-\theta)[\theta(\mu+p-1)-p]/\theta} + C
	\end{equation*}
	due to the assumption \eqref{assump-2}. While $\widehat{C}(t)$ is continuous we obtain finally
	\begin{equation*}
		\sum_{i=1}^{N}\|v_i(T)\|_p^p \leq \sum_{i=1}^{N}\|v_{i,0}\|_p^p + C\int_0^T\widehat C(t)^{(1-\theta)[\theta(\mu+p-1)-p]/\theta}dt + CT < +\infty,
	\end{equation*}
	for arbitrary $p\geq p_0+\delta$, which together with Lemma \ref{heat-regularity-2} implies the global existence of classical solutions. The proof of Proposition \ref{pro:extension} is thus complete.
\end{proof}

Note that Proposition \ref{pro:extension} provides only the global existence of classical solutions but not the uniform-in-time bound as in Theorem \ref{theo:main}. The reason is that, in contrast to the case of small $L^2$-initial data, one does not have in general the convergence of solution $u$ to the complex balanced equilibrium $u_{\infty}$, or equivalently the decay of $v = u - u_{\infty}$ to zero, except the case when the considered system is complex balanced and has no boundary equilibria, then the solutions are proved to converge to $u_{\infty}$ in $L^1$-norm without the assumption of initial data being close to $u_{\infty}$ (see e.g. \cite{DFT16,FT17,PSZ16}).
\begin{corollary}
	Assume that the system \eqref{system} is complex balanced and has no boundary equilibria. Suppose that all the assumptions in Proposition \ref{pro:extension} are satisfied with the function $\widehat{C}(t)$ in \eqref{assump-2} grows at most polynomially. Assume moreover that the domain $\Omega$ is regular enough, i.e. $\partial\Omega$ is of class $C^{\infty}$. Then the classical solution to \eqref{system} exists globally and converges to the complex balanced equilibrium in $L^{\infty}$-norm.
\end{corollary}
\begin{proof}
	By Proposition \ref{pro:extension} and the assumption that $\widehat{C}(t)$ grows at most polynomially, we obtain for the classical solution to \eqref{system} that for all $i=1,\ldots, N$
	\begin{equation*}
		\|u_i(T)\|_{\infty} \leq C_T
	\end{equation*}
	and consequently, 
	\begin{equation*}
		\|u_i(T)\|_{H^k(\Omega)} \leq C_T \text{ for all } T\geq t_0 > 0,
	\end{equation*}
	for any $k\in \mathbb N$ and for some $t_0>0$, where $C_T$ is a constant grows at most polynomially in $T$, thanks to the regularity of $\Omega$ and arguments similarly to the proof of Theorem \ref{theo:main}. Since the system is assumed to have no boundary equilibria, it was proved in \cite{DFT16,FT17} that
	\begin{equation*}
		\sum_{i=1}^{N}\|u_i(t) - u_{i,\infty}\|_{1} \leq Ce^{-\lambda t} \quad \text{ for all } \quad t>0.
	\end{equation*}
	Therefore an application of the Gagliardo-Nirenberg inequality concludes the Corollary.
\end{proof}

\noindent{\bf Acknowledgements.} We would like to thank the referee for helpful suggestions which improve the presentation of the paper.

This work is partially supported by International Training Program IGDK 1754 and NAWI Graz.

\end{document}